\documentclass[a4paper,10pt]{amsart}
\usepackage{amssymb}
\usepackage[all]{xy}
\usepackage{mathtools}
\usepackage{enumitem}
\usepackage{tikz}
\usepackage{tikz-cd}
\usetikzlibrary{positioning}
\usetikzlibrary{matrix,arrows,decorations.pathmorphing}
\def\cprime{$'$} \def\polhk#1{\setbox0=\hbox{#1}{\ooalign{\hidewidth
 \lower1.5ex\hbox{`}\hidewidth\crcr\unhbox0}}}
\newtheorem{prop}{Proposition}
\newtheorem{thm}[prop]{Theorem}
\newtheorem{cor}[prop]{Corollary}
\newtheorem{lem}[prop]{Lemma}

\newtheorem{defn}{Definition}[section]
\newtheorem{ques}{Question}

\theoremstyle{definition}

\theoremstyle{remark}
\newtheorem{rem}[prop]{Remark}

\numberwithin{prop}{section} 
\numberwithin{claim}{prop}
\numberwithin{equation}{section}

\newcommand{\kernel}{\mathrm{ker}}
\newcommand{\coker}{\mathrm{coker}}

\newcommand{\ccd}{\mathrm{cd}}

\newcommand{\nclose}[1]{\ensuremath{\langle\!\langle#1\rangle\!\rangle}}

\DeclareMathOperator{\asdim}{\mathrm{asdim}}

\newcommand{\rst}{\mathrm{res}}

\newcommand{\FP}{\textup{FP}}

\newcommand{\ca}[1]{\mathcal{#1}}

\newcommand{\Z}{\mathbb{Z}}
\newcommand{\Q}{\mathbb{Q}}
\newcommand{\F}{\mathbb{F}}
\newcommand{\N}{\mathbb{N}}

\newcommand{\QG}{\Q[G]}
\newcommand{\QH}{\Q[H]}

\newcommand{\dis}{\mathbf{dis}}
\newcommand{\Mod}{\mathbf{mod}}

\newcommand{\QGmod}{{}_{\QG}\Mod}

\newcommand{\QGdis}{{}_{\QG}\dis}
\newcommand{\QHdis}{{}_{\QH}\dis}

\newcommand{\argu}{\hbox to 7truept{\hrulefill}}

\DeclarePairedDelimiter{\abs}{\lvert}{\rvert}

\DeclareMathOperator{\aut}{Aut}
\DeclareMathOperator{\dist}{dist}
\newcommand{\norm}[1]{\left\lVert#1\right\rVert}

\title[]{Subgroups, Hyperbolicity and Cohomological Dimension for Totally Disconnected Locally Compact Groups}

\author{S.~Arora}
\author{I.~Castellano}
\author{G.~Corob Cook}
\author{E.~Mart\'inez-Pedroza}
\email{sarora17@mun.ca}
\email{ilaria.castellano88@gmail.com}
\email{gcorobcook@gmail.com}
\email{eduardo.martinez@mun.ca}


\begin{document}

\begin{abstract} 
This article is part of the program of studying large-scale geometric properties of totally disconnected locally compact groups, TDLC-groups, by analogy with the theory for discrete groups. We provide a characterization of hyperbolic TDLC-groups, in terms of homological isoperimetric inequalities. This characterization is used to prove the main result of the article: for hyperbolic TDLC-groups with rational discrete cohomological dimension $\leq 2$, hyperbolicity is inherited by compactly presented closed subgroups. As a consequence, every compactly presented closed subgroup of the automorphism group $\aut(X)$ of a negatively curved locally finite $2$-dimensional building $X$ is a hyperbolic TDLC-group, whenever $\aut(X)$ acts with finitely many orbits on $X$. Examples where this result applies include hyperbolic Bourdon's buildings. 

We revisit the construction of small cancellation quotients of amalgamated free products, and verify that it provides examples of hyperbolic TDLC-groups of rational discrete cohomological dimension $2$ when applied to amalgamated products of profinite groups over open subgroups. 

We raise the question of whether our main result can be extended to locally compact hyperbolic groups if rational discrete cohomological dimension is replaced by asymptotic dimension. We prove that this is the case for discrete groups and sketch an argument for TDLC-groups.
\end{abstract}

	\maketitle 
\section{Introduction}
A locally compact group $G$ is said to be {\it totally disconnected} if the identity is its own connected component. For an arbitrary locally compact group, the connected component of the identity is always a closed normal subgroup with a totally disconnected quotient. Therefore, in principle, the study of locally compact groups can be reduced to the study of the two subclasses formed by connected locally compact groups and totally disconnected locally compact groups.  By the celebrated solution to Hilbert's fifth problem, connected locally compact groups are inverse limits of Lie groups.  However, such a thorough understanding has not been achieved for the totally disconnected counterpart.

Hereafter, we use TDLC-group as a shorthand for totally disconnected locally compact group. The class of TDLC-groups has been a topic of interest in the last three decades since the work of G. Willis~\cite{Wi94}, and of M. Burger and S. Mozes~\cite{BuMo97}.

Large-scale properties of a TDLC-group $G$ can be addressed by investigating a family of quasi-isometric locally finite connected graphs which are known as {Cayley-Abels graphs of $G$}; see \S~\ref{defn:CAgraph} for the definition and further details. Therefore, the theory of TDLC-groups becomes amenable to many tools from geometric group theory (see ~\cite{Bau, BSW08, Mo02} for example) and the notion of hyperbolic group carries over to the realm of TDLC-groups.

The motivation for this work is to gain a better understanding of the interaction between the geometric properties of the TDLC-group $G$ and its cohomological properties by analogy with the discrete case. An investigation of this type was initiated in~\cite{CaWe16, Ca18} where the {\em rational discrete cohomology} for TDLC-groups has been introduced and the authors have shown that many well-known properties that hold for discrete groups can be transferred to the context of TDLC-groups (in some cases after substantial work). 

For a TDLC-group $G$, the representation theory used in~\cite{CaWe16} leans on the notion of {\em discrete $\QG$-module}, that is a $\QG$-module $M$ such that the action $G\times M\to M$ is continuous when $M$ carries the discrete topology. In the case that $G$ is discrete, any $\QG$-module is discrete. Because of the divisibility of $\Q$, the abelian category $\QGdis$ of discrete $\QG$-modules has enough projectives. As a consequence, the notions of {\em rational discrete cohomological dimension}, denoted by $\ccd_\Q(G)$, and {\it type $\FP_n$} can be introduced for every TDLC-group $G$ in the category $\QGdis$ (see \S\ref{ss:cfc} for the necessary background). This opens up the possibility of investigating TDLC-groups by imposing some cohomological finiteness conditions.

The main result of this article is a subgroup theorem for hyperbolic TDLC-groups of rational discrete cohomological dimension at most $2$.
\begin{thm}\label{thm:gersten1}
Let $G$ be a hyperbolic TDLC-group with $\ccd_\Q(G) \leq 2 $. Every compactly presented closed subgroup $H$ of $G$ is hyperbolic.
\end{thm}
This theorem generalizes the following two results for discrete groups:
\begin{itemize}
\item Finitely presented subgroups of hyperbolic groups of integral cohomological dimension less than or equal to two are hyperbolic. This is a result of Gersten~\cite[Theorem 5.4]{Ge96} which can be recovered as a consequence of the inequality $\ccd_\Q(\argu)\leq\ccd_\Z(\argu)$.

\item Finitely presented subgroups of hyperbolic groups of rational cohomological dimension less than or equal to two are hyperbolic. This is a recent result in~\cite{AMP19} which is the analogue of Theorem~\ref{thm:gersten1} in the discrete case.
\end{itemize}

We remark that Brady constructed an example of a discrete hyperbolic group of integral cohomological dimension three that contains a finitely presented subgroup that is not hyperbolic~\cite{Br99}. Hence the the dimensional bound on the results stated above is sharp. 

In the discrete case, a class of hyperbolic groups of rational cohomological dimension two is given by groups admitting finite presentations with certain small cancellation conditions. This is also the case for TDLC-groups for small cancelation quotiens of amalgamated products of profinite groups. We refer the reader to Section~\ref{sec:SmallCancellation} for details on the following result.

\begin{thm}
Let $A \ast_C B$ be the amalgamated free product of the profinite groups $A,B$ over a common open subgroup $C$. Let $R$ be a finite symmetrized subset of $A\ast_C B$ that satisfies the $C'(1/12)$ small cancellation condition. Then the quotient $G= (A \ast_C B)/\nclose{R}$ is a hyperbolic TDLC-group with  $\ccd_\Q (G) \leq 2$.
 \end{thm}

In the framework of discrete groups, it is a result of Gersten that type $\FP_2$ (over $\Z$) subgroups of hyperbolic groups of integral cohomological dimension at most two are hyperbolic~\cite[Theorem 5.4]{Ge96}. We raise the following question:
\begin{ques} 
Does Theorem~\ref{thm:gersten1} remain true if $H$ is of type $\FP_2$ but not compactly presented?
\end{ques}

It is well known that if $X$ is a locally finite simplicial complex then the group of simplicial automorphisms $\aut(X)$ endowed with the compact open topology is a TDLC-group~\cite[Theorem~2.1]{Cameron96}. If, in addition, $X$ admits a $CAT(-1)$ metric and $\aut(X)$ acts with finitely many orbits on $X$, then $\aut(X)$ is a hyperbolic TDLC-group with $\ccd_\Q(\aut(X)) \leq \dim(X) $. 
\begin{cor}\label{cor:CAT}
Let $X$ be a locally finite $2$-dimensional simplicial $CAT(-1)$-complex. If $\aut(X)$ acts with finitely many orbits on $X$, then every compactly presented closed subgroup of $\aut(X)$ is a hyperbolic TDLC-group.
\end{cor}
A discrete version of Corollary~\ref{cor:CAT} was proved in~\cite[Corollary 1.5]{HM14} using combinatorial techniques. There are different sources of complexes $X$ satisfying the hypothesis of Corollary~\ref{cor:CAT} and such that $\aut(X)$ is a non-discrete TDLC-group. For example:
\begin{itemize}
\item Bourdon's building $I_{p,q}$, $p\geq 5$ and $q\geq 3$, is the unique simply connected polyhedral $2$-complex such that all $2$-cells are right-angled hyperbolic $p$-gons and the link of each vertex is the complete bipartite graph $K_{q,q}$. These complexes were introduced by Bourdon~\cite{Bo97}. The natural metric on $I_{p,q}$ is $CAT(-1)$ and $\aut(I_{p,q})$ is non-discrete. 
\item For an integer $k$ and a finite graph $L$, a $(k, L)$-complex is a simply connected $2$-dimensional polyhedral complex such that all $2$-dimensional faces are $k$-gons and the link of every vertex is isomorphic to the graph $L$. A result of \'{S}wi\polhk{a}tkowski~\cite[Main Theorem (1)]{Sw98} provides sufficient conditions on the graph $L$ guaranteeing that if $k\geq 4$ then $\aut(X)$ is a non-discrete group for any $(k,L)$-complex $X$. It is a consequence of Gromov's link condition, that a $(k,L)$-complex admits a $CAT(-1)$-structure for any $k$ sufficiently large. 
\end{itemize}
In order to prove Theorem~\ref{thm:gersten1}, we follow ideas from Gersten~\cite{Ge96}. We introduce the concept of {\it weak $n$-dimensional linear isoperimetric inequality} for TDLC-groups, which is a homological analogue in higher dimensions of linear isoperimetric inequalities. Profinite groups are characterized as TDLC-groups satisfying the weak $0$-dimensional linear isoperimetric inequality: see~Section~\ref{sec:weakiso}. The weak $1$-dimensional linear isoperimetric inequality is called from here on the \emph{weak linear isoperimetric inequality}. The following result generalizes for TDLC-groups a well-known characterization of hyperbolicity in the discrete case~\cite[Theorem 3.1]{Ge96}.
 \begin{thm}\label{thm:characterization}
A compactly generated TDLC-group $G$ is hyperbolic if and only if $G$ is compactly presented and satisfies the weak linear isoperimetric inequality.
\end{thm}
The property of satisfying the weak $n$-dimensional linear isoperimetric inequality is inherited by closed subgroups under some cohomological finiteness conditions.
\begin{thm}\label{thm:mainineq1}
Let $G$ be a TDLC-group of type $\FP_\infty$ with $\ccd_\Q(G) = n + 1$ that satisfies the weak $n$-dimensional linear isoperimetric inequality. Then every closed subgroup $H$ of $G$ of type $\FP_{n+1}$ satisfies the weak $n$-dimensional linear isoperimetric inequality.
\end{thm}
The major part of the paper is devoted to the proof of Theorem~\ref{thm:mainineq1}.
Our strategy borrows ideas from~\cite{AMP19, Ge96, HM16}. Some remarks:
\begin{itemize}
\item In the case that $G$ is discrete, Theorem~\ref{thm:mainineq1} is a consequence of~\cite[Theorem 1.7]{AMP19}.
\item The arguments in~\cite{AMP19}, where the authors replace some topological techniques from~\cite{Ge96, HM16} with algebraic counterparts, carry over the TDLC case only under the stronger assumption that the subgroup $H$ is open, see Remark~\ref{rem:openvsclosed}.    
\item Currently, for TDLC-groups, there is no well studied notion of $n$-dimensional homological Dehn function as the definitions available in the discrete case, see for example~\cite{ABDY13, HM16}. In contrast to the arguments in~\cite{AMP19}, we avoid the use of these objects and provide a more straight forward argument.  
\end{itemize}
It is a simple verification that Theorem~\ref{thm:gersten1} follows by Theorems~\ref{thm:mainineq1} and~\ref{thm:characterization}.
\begin{proof}[Proof of the Theorem~\ref{thm:gersten1}]
Since $G$ is hyperbolic, Theorem~\ref{thm:characterization} implies that $G$ sati\-sfies the weak linear isoperimetric inequality. By Theorem~\ref{thm:mainineq1}, $H$ also satisfies the weak linear isoperimetric inequality. We can then apply Theorem~\ref{thm:characterization} again to conclude the proof.
\end{proof}

\subsection*{Locally compact hyperbolic groups}

The monograph~\cite{CoHa16} by de Cornulier and de la Harpe  laid out the foundations of the study of locally compact groups from the perspective of geometric group theory. In this context,  a locally compact group is hyperbolic if it has a continuous proper cocompact isometric action on some proper
geodesic hyperbolic metric space~\cite{CCMT15}; this generalizes the classical definition in the discrete case as well as the definition in the class of totally disconnected locally compact groups used in the present article. By analogy with the discrete case, the asymptotic dimension provides a quasi-isometry invariant of locally compact compactly generated groups. The question below  suggests a posible generalization of Theorem~\ref{thm:gersten1} for the larger class of locally compact hyperbolic groups.  

\begin{ques}\label{q:LChypAsdim}
Let $G$ be a locally compact hyperbolic group such that $\asdim G \leq 2$. Are compactly presented subgroups of $G$ hyperbolic?
\end{ques}


We conclude the introduction of the article by verifying that the  the above question has a positive answer for discrete hyperbolic groups. The argument provides a blueprint to answer the question in the positive in class of hyperbolic TDLC-groups using Theorem~\ref{thm:gersten1}. 

\begin{thm}\label{thm:GerstenAsdim}
Let $G$ be a discrete hyperbolic group such that $\asdim G\leq 2$. Then every finitely presented subgroup of $G$ is hyperbolic.
\end{thm}
\begin{proof}
The main result of~\cite{AMP19} states that if $\ccd_\Q G \leq 2$, then any finitely presented subgroup of $G$ is hyperbolic. Therefore it is enough to verify  the inequality \[ \ccd_\Q G  \leq \asdim G.\] 
This inequality relies on important work of Buyalo and Bedeva~\cite{BL90}  and Bestvina and Mess~\cite{BM91} as   explained below. 

It is a result of Buyalo and Lebedeva,~\cite[Theorem 6.4]{BL90}, that the asymptotic dimension of every cobounded, hyperbolic, proper, geodesic metric space $X$ equals the topological dimension of its boundary at infinity plus $1$,
\[ \asdim X =\dim \partial_\infty X +1 .\]

On the other hand, for a compact metrizable space $Y$, there is a notion of cohomological dimension $\dim_R Y$ over a ring $R$. It is known that if $\dim Y < \infty$  then 
\[\dim_\Q Y \leq \dim_\Z Y = \dim Y ,\]
see~\cite{BM91} for definitions and references. 

Let  $\partial_\infty G$ denote the Gromov boundary of $G$. Recall that $\partial_\infty G$ is a compact metrizable space with finite topological dimension, see for example~\cite{KB02}. It follows that 
\[ \dim_\Q \partial_\infty G \leq \dim \partial_\infty G.\]
The work of Bestvina and Mess~\cite[Corollary 1.4]{BM91} implies that if $\ccd_\Q G<\infty$ then 
\begin{equation}\label{eq:BestvinaMess} \ccd_\Q G = \dim_\Q \partial_\infty G + 1. \end{equation}
Since discrete hyperbolic groups admit finite dimensional models for the universal space for proper actions (Rips complexes with large parameter, see~\cite{MS02} or~\cite{HOP14}), it follows that $\ccd_\Q G<\infty$.  Therefore
$ \ccd_\Q G  \leq \asdim G.$
\end{proof}
\begin{rem}
We expect a positive answer to Question~\ref{q:LChypAsdim} for hyperbolic TDLC-groups. Indeed, to obtain a positive answer it is enough to verify the following statement generalizing work of Bestvina and Mess~\cite[Corollary 1.4]{BM91}:
\begin{itemize}[leftmargin=*]
\item[-] Let $G$ be a hyperbolic TDLC-group.  If $\ccd_\Q G<\infty$ then 
$\ccd_\Q G = \dim_\Q \partial_\infty G + 1$,
where $\ccd_\Q G$ is the rational discrete cohomological dimension.
 \end{itemize}
Then the proof of Theorem~\ref{thm:GerstenAsdim} works in the TDLC case by using  Theorem~\ref{thm:gersten1} and that hyperbolic TDLC-groups have finite rational discrete cohomological dimension, see Proposition~\ref{prop:F2model}. An attempt to generalize the work of Bestvina and Mess in~\cite{BM91} for TDLC-groups is currently work in progress by the second author, F.W. Pasini and T. Weigel. In the generality of locally compact groups, the authors are not aware of a cohomology theory for locally compact groups that allow to persuade the techniques of this article.
\end{rem}

\subsection*{Organization}
Preliminary definitions regarding TDLC-groups and rational discrete modules are given in Section~\ref{sec:prilim}. Then Section~\ref{sec:CayleyAbels} consists of definitions and some preliminary results on Cayley-Abels graphs, compact presentability and hyperbolicity for TDLC-groups. Section~\ref{sec:weakiso} introduces the weak $n$-dimensional linear isoperimetric inequality. Section~\ref{sec:proof1} is devoted to the proof of Theorem~\ref{thm:mainineq1}. Finally, Section~\ref{sec:proof2} relates hyperbolicity and the weak linear isoperimetric inequality and contains the proof of Theorem~\ref{thm:characterization}.

\subsection*{Acknowledgements}
The authors thank David Futer for consultation on the theory of hyperbolic buildings. I.C. and E.M.P. thank the organisers of the satellite meetings \emph{Geometric Group Theory in Campinas} and \emph{Group Theory, Cabo Frio, Rio de Janeiro} of the 2018 International Congress of Mathematicians in Brazil, where the results of this article were first envisioned and initial progress was made. I.C. was partially supported by EPSRC grant EP/N007328/1 and the FA project ``Strutture Algebriche'' ATE-2016-0045. G.C.C. was supported by ERC grant 336983 and Basque government grant IT974-16. E.M.P. acknowledges funding by the Natural Sciences and Engineering Research Council of Canada, NSERC.

\section{TDLC-groups and rational discrete $G$-modules}\label{sec:prilim}
Throughout this section $G$ always denotes a TDLC-group.
Note that a TDLC-group is Hausdorff. Discrete groups are TDLC-groups. Profinite groups are precisely compact TDLC-groups~\cite[Proposition 0]{Se02}. 
A fundamental result about the structure of TDLC-groups is known as van Dantzig's Theorem:
\begin{thm}[van Dantzig's Theorem, \cite{Vd36}] The family of all compact open subgroups of a TDLC-group $G$ forms a neighbourhood system of the identity element. 
\end{thm}
Note that every Hausdorff topological group admitting such a local basis is necessarily TDLC. Hence the conclusion of van Dantzig's Theorem characterizes TDLC-groups in the class of Hausdorff topological groups. 

For example, the non-Archimedean local fields $\Q_p$ and $\F_q((t))$ admit, respectively, the following local basis at the identity element:
	\begin{enumerate}
		\item $\{p^n\Z_p\mid n\in\N\}$, where $\Z_p=\{x\in\Q_p\mid |x|\leq1\}=\{x\in\Q_p\mid |x|<p\}$ is compact and open;
		\item $\{t^n\F_q[[t]]\mid n\in\N\}$, where the norm is defined by $q^{-ord(f)}$.
	\end{enumerate}

\subsection{Rational discrete $G$-modules}
Let $\Q$ denote the field of rational numbers, and let $\QGmod$ be the category of abstract left $\QG$-modules and their homomorphisms. A left $\QG$-module $M$ is said to be {\it discrete} if the stabilizer
\begin{equation}\nonumber
G_m=\{g\in G\mid g\cdot m=m\},
\end{equation}
of each element $m\in M$ is an open subgroup of $G$. Equivalently, the action $G\times M \to M$ is continuous when $M$ carries the discrete topology. The full subcategory of $\QGmod$ whose objects are discrete $\QG$-modules is denoted by $\QGdis$. It was shown in~\cite{CaWe16} that $\QGdis$ is an abelian category with enough injectives and projectives.

\subsection{Permutation $\QG$-modules in $\QGdis$} 

Let $\Omega$ be a non-empty left $G$-set. For $\omega\in \Omega$ let $G_\omega$ denote the pointwise stabilizer. The $G$-set $\Omega$ is called {\em discrete} if all pointwise stabilizers  are open subgroups of $G$, and $\Omega$ is called {\em proper} if all pointwise stabilizers  are open and compact. 

The $\Q$-vector space $\Q[\Omega]$ - freely spanned by a discrete $G$-set $\Omega$ - carries a canonical structure of discrete left $\QG$-module called the {\em discrete permutation $\QG$-module induced by $\Omega$}. 

Note that a discrete permutation $\QG$-module in $\QGdis$ is a coproduct 
\begin{equation}\nonumber
\Q[\Omega]\cong\coprod_{\omega\in\ca R}\Q[G/G_\omega],
\end{equation}
in $\QGdis$, where $\ca R$ is a set of representatives of the $G$-orbits in $\Omega$, and $\Omega$ is a discrete $G$-set.

A {\em proper permutation $\QG$-module} is a discrete $\QG$-module of the form $\Q[\Omega]$ where $\Omega$ is a proper $G$-set.

A proper permutation $\QG$-module is a projective object in $\QGdis$; see~\cite{CaWe16}. The arguments of this article rely on the following characterization of projective objects in $\QGdis$, a non-trivial result that in particular relies on Maschke's theorem on irreducible representations of finite groups, and Serre's results on Galois cohomology.
\begin{prop}[\protect{\cite[Corollary~3.3]{CaWe16}}]\label{prop:proj} Let $G$ be a TDLC-group. A discrete $\QG$-module $M$ is projective in $\QGdis$ if, and only if, $M$ is a direct summand of a proper permutation $\QG$-module in $\QGdis$.
\end{prop}
Throughout the article, we only consider resolutions consisting of discrete permutation $\QG$-modules, and we refer to this type of resolutions as {\em permutation resolutions in $\QGdis$}. Analogously, a resolution that consists only of proper permutation modules is called a {\em proper permutation resolution in $\QGdis$}. When the category is clear from the context, we will omit the term ``in $\QGdis$".

\subsection{Rational discrete homological finiteness}\label{ss:cfc}
Following~\cite{CaWe16}, we say that a TDLC-group $G$ is of {\em type $\FP_n$} ($n\in\N$) if there exists a partial proper permutation resolution in $\QGdis$
\begin{equation}\label{eq:fp}
\xymatrix{\Q[\Omega_n]\ar[r]&\Q[\Omega_{n-1}]\ar[r]&\cdots\ar[r]&\Q[\Omega_0]\ar[r]&\Q\ar[r]&0}
\end{equation}
of the trivial discrete $\QG$-module $\Q$ of {\it finite type}, i.e., every discrete left $G$-set $\Omega_i$ is finite modulo $G$ or equivalently $\Q[\Omega_i]$ is finitely generated. Type $\FP_n$ in this paper will always mean over $\Q$, though the definition generalizes to finite type proper permutation resolutions over discrete rings other than $\Q$, where the proper permutation modules are no longer projective in general -- see for example \cite{CCC}. The group $G$ is of {\em type $\FP_\infty$} if it is $\FP_n$ for every $n\in \N$. Notice that having type $\FP_0$ is an empty condition for a TDLC-group $G$. On the other hand, having type $\FP_1$ is equivalent to be compactly generated (see \cite[Proposition~5.3]{CaWe16}) and compact presentation implies type $\FP_2$. 

The {\em rational discrete cohomological dimension} of $G$, $\ccd_\Q(G)\in\N\cup\{\infty\}$, is defined to be the minimum $n$ such that the trivial discrete $\QG$-module $\Q$ admits a projective resolution
\begin{equation}
\xymatrix{0\ar[r]&P_n\ar[r]^{\partial_n}&P_{n-1}\ar[r]&\cdots\ar[r]&P_0\ar[r]&\Q\ar[r]&0}
\end{equation}
in $\QGdis$ of length $n$. The rational discrete cohomological dimension reflects structural
information on a TDLC-group $G$. For example, $G$ is profinite if and only if $\ccd_\Q(G)=0$.

By composing the notions above, one says that $G$ is of {\em type $\FP$} if
\begin{enumerate}[label=(\roman*)]
\item $G$ is of type $\FP_\infty$, and
\item $\ccd_\Q(G)=d<\infty$.
\end{enumerate}
For a TDLC-group $G$ of type $FP$, the trivial left $\QG$-module $\Q$ possesses a projective resolution $(P_{\bullet},\partial_{\bullet})$ which is finitely generated and concentrated in degrees $0$ to $d$. It is not known whether $(P_{\bullet},\partial_{\bullet})$ can be assumed to be a proper permutation resolution of finite length.

\subsection{Restriction of scalars}\label{ss:resind}
Let $H$ be a closed subgroup of the TDLC-group $G$. It follows that $H$ is a TDLC-group and in particular the category $\QHdis$ is well defined.
The restriction of scalars from $\QG$-modules to $\QH$-modules preserves discretness. In other words there is a well defined {\it restriction functor}
\begin{equation}
\rst^G_H(\argu)\colon\QGdis\to\QHdis,
\end{equation}
obtained by restriction of scalars via the natural map $\QH\hookrightarrow\QG$. The restriction is an exact functor which maps projectives to projectives. Indeed, for every proper permutation $\QG$-module $\Q[\Omega]$, the discrete $\QH$-module $\rst^G_H(\Q[\Omega])$ is still a proper permutation module in $\QHdis$. To simplify notation, for a discrete $\QG$-module $M$, we may write $M$ for $\rst^G_H(M)$ when the meaning is clear.
\medskip

\section{Cayley-Abels graphs, Compact presentability and Hyperbolicity}\label{sec:CayleyAbels}

\subsection{Compactly generated TDLC-groups and Cayley-Abels graphs}\label{defn:CAgraph}

In this article a graph is a $1$-dimensional simplicial complex, hence graphs are undirected, without loops, and without multiple edges between the same pair of vertices. 

A locally compact group is said to be {\em compactly generated} if there exists a compact subset that algebraically generates the whole group. 

\begin{prop}\cite[Theorem 2.2]{KrMo08}
A TDLC-group $G$ is compactly generated if and only if it acts vertex transitively with compact open vertex stabilizers  on a locally finite connected graph $\Gamma$.
\end{prop}

A graph with a $G$-action as in the proposition above is called a {\em Cayley-Abels graph} for $G$. In~\cite{KrMo08} these graphs are referred to as {\em rough Cayley graphs} but the notion of Cayley-Abels graph traces back to Abels~\cite{Ab72}. 

As soon as the compactly generated TDLC-group $G$ is non-discrete, the $G$-action on a Cayley-Abels graph is never free. That is to say, the action always has non-trivial vertex stabilizers . Nevertheless, these large but compact stabilizers  play an important role in the study of the cohomology of $G$: they give rise to proper permutation $\QG$-modules. 

A consequence of van Dantzig's Theorem is the following.

\begin{prop}\label{prop:compactgeneration}
For a TDLC-group $G$ the following statements are equivalent:
\begin{enumerate}
\item $G$ is compactly generated.
\item There exists a compact open subgroup $K$ of $G$ and a finite subset $S$ of $G$ such that $K\cup S$ generates $G$ algebraically.
\item There exists a finite graph of profinite groups $(A, \Lambda)$ with a single vertex, together with a continuous open surjective homomorphism $\phi\colon \pi_1(A, \Lambda, \Xi) \to G$ such that $\phi|_{\mathcal A_v}$ is injective for all $v\in \mathcal V(\Lambda)$.
\end{enumerate}
\end{prop}
\begin{proof}
Note that if $C$ is a compact set generating $G$ and $K$ is a compact open subgroup of $G$ then there a finite subset $S\subset G$ such that the collection of left cosets $\{sK | s\in S\}$ covers $C$. Hence, by van Dantzig's Theorem, (1) implies (2). 
To show that (2) implies (3), consider the graph of groups with a single vertex and an edge for each element of $S$. The vertex group is $K$, and each edge group is $K\cap K^s$ with morphisms the inclusion and conjugation by $s$: see~\cite[Proposition 5.10, proof of (a)]{CaWe16}. That (3) implies (1) is immediate since $G$ is a quotient of the compactly generated TDLC-group $\pi_1(A, \Lambda, \Xi)$.
\end{proof}
Note that in the terminology of the third statement of the above proposition, a Cayley-Abels graph for $G$ can be obtained by considering the quotient of the (topological realisation as a $1$-dimensional simplicial complex of the) universal tree of $(A,\Lambda)$ by the kernel of $\phi$. 

\subsection{Quasi-isometry for TDLC-groups and Hyperbolicity.} 
 
The edge-path metric on a Cayley-Abels graph $\Gamma$ of a TDLC-group $G$ induces a left-invariant pseudo-metric on $G$, by pulling back the metric of the $G$-orbit of a vertex of $\Gamma$. In the following proposition, we denote this pseudo-metric by $\dist_\Gamma$. 
 
Following~\cite{CoHa16}, an action of a topological group $G$ on a (pseudo-) metric space $X$ is \emph{geometric} if it satisfies:
\begin{itemize}
\item (Isometric) The action is by isometries; 
\item (Cobounded) There is $F\subset X$ of finite diameter such that $\bigcup_{g\in G} gF =X$;
 \item (Locally bounded) For every $g\in G$ and bounded subset $B\subset X$ there is a neighborhood $V$ of $g$ in $G$ such that $VB$ is bounded in $X$; and 
 \item (Metrically proper) The subset $\{g\in G\colon \dist_X(x , gx)\leq R \}$ is relatively compact in $X$ for all $x\in X$ and $R>0$. 
 \end{itemize}
 
The following version of the {\v S}varc-Milnor Lemma is a consequence of work by Cornulier and de la Harpe on locally compact groups; see~\cite[Corollary 4.B.11 and Theorem 4.C.5]{CoHa16}.
\begin{prop}\label{prop:svarc}
 Let $G$ be a TDLC-group, let $X$ be a geodesic (pseudo-) metric space, and let $x\in X$. Suppose there exists a geometric action of $G$ on $X$. Then there is a Cayley-Abels graph $\Gamma$ for $G$ such that the map between the pseudo-metric spaces 
 \[ (G, \dist_\Gamma) \to (X, \dist_X),\qquad x\mapsto gx \]
 is a quasi-isometry. 
 \end{prop}

This proposition implies the following result from~\cite[Theorem 2.7]{KrMo08}. 
\begin{cor}
The Cayley-Abels graphs associated to a compactly generated TDLC-group are all quasi-isometric to each other.
\end{cor}

 This quasi-isometric invariance of Cayley-Abels graphs allows us to define geometric notions for compactly generated TDLC-groups such as ends, number of ends or growth, by considering quasi-isometric invariants of a Cayley-Abels graph associated to $G$. 
 
 \begin{defn}
A TDLC-group $G$ is defined to be {\it hyperbolic} if $G$ is compactly generated and some (hence any) Cayley-Abels graph of $G$ is hyperbolic. 
\end{defn}

For an equivalent definition of hyperbolic TDLC-group using (standard) Cayley graphs over compact generating sets see~\cite{BMW12} for details.

\subsection{Compactly presented TDLC-groups}\label{subsec:compactlypresented}

A locally compact group is said to be {\em compactly presented} if it admits a presentation $\langle K\mid R\rangle$ where $K$ is a compact subset of $G$ and there is a uniform bound on the length of the relations in $R$. Observe that being compactly presented implies being compactly generated. There are also an equivalent definition of compact presentation~\cite[\S~5.8]{CaWe16} based on van Dantzig's Theorem in the context of Proposition~\ref{prop:compactgeneration}.

\begin{cor}\cite{CoHa16}
A TDLC-group $G$ is {compactly presented} if and only if
\begin{enumerate}
\item there exists a finite graph of profinite groups $(A, \Lambda)$ with a single vertex together with a continuous open surjective homomorphism $\phi\colon \pi_1(A, \Lambda, \Xi) \to G$ such that $\phi|_{\mathcal A_v}$ is injective for all $v\in \mathcal V(\Lambda)$, and 
\item the kernel of $\phi$ is finitely generated as a normal subgroup.
\end{enumerate}
\end{cor}
\begin{proof}
Note that the if direction is immediate since $ \pi_1(A, \Lambda, \Xi)$ is compactly presented. Indeed, a group presentation of $\pi_1(A, \Lambda, \Xi)$ has as generators the formal union of the vertex group and a finite number of elements corresponding to the edges of the graph. The set of relations consists of the multiplication table of the vertex group and the HNN-relations; note that all these relations have length at most four. Since the kernel of $\phi$ is finitely generated as a normal subgroup, it follows that $G$ is compactly presented.

For the only if direction, since $G$ is compactly presented, in particular it is compactly generated and hence there is a finite graph of profinite groups $(A, \Lambda)$ with the required properties for (1). It remains to show that the kernel of $\phi$ is finitely generated as a normal subgroup. By \cite[Proposition 5.10(b)]{CaWe16}, $\kernel(\phi)$ is a discrete subgroup of $\pi_1(A, \Lambda, \Xi)$. Since $\pi_1(A, \Lambda, \Xi)$ is compactly generated and $G$ is compactly presented, \cite[Proposition 8.A.10(2)]{CoHa16} implies that $\kernel(\phi)$ is compactly generated as a normal subgroup; by discreteness it follows that $\kernel(\phi)$ is finitely generated as a normal subgroup.
\end{proof}

\begin{prop}\label{prop:F2model0}
A TDLC-group $G$ is compactly presented if and only if there exists a simply connected cellular $G$-complex $X$ with compact open cell stabilizers , finitely many $G$-orbits of cells of dimension at most $2$, and such that elements of $G$ fixing a cell setwise fixes it pointwise (no inversions). 
\end{prop}
 A $G$-complex with the properties stated in the above proposition is called a \emph{topological model of $G$ of type $\F_2$}. 
 \begin{proof}[Proof of Proposition~\ref{prop:F2model0}] The equivalence of compact presentability and the existence of a topological model for $\F_2$ follows from standard arguments. That compact presentability is a consequence of the existence of the topological model follows directly from~\cite[I.8, Theorem 8.10]{BH99}; for compact presentability implying the existence of such a complex see for example~\cite[Proposition 3.4]{CCC}. 
\end{proof}

The following result is well known for discrete hyperbolic groups. The proof in~\cite[III.$\Gamma$ Theorem 3.21]{BH99} carries over for hyperbolic TDLC-groups by considering the Rips complex on a Cayley-Abels graph instead of the standard Cayley graph. 

\begin{prop}\label{prop:F2model}
Let $G$ be a hyperbolic TDLC-group. Then $G$ acts on a simplicial complex $X$ such that:
\begin{enumerate}
\item $X$ is finite dimensional, contractible and locally finite;
\item $G$ acts simplicially, cell stabilizers  are compact open subgroups, and there are finitely many $G$-orbits of cells.
\item $G$ acts transitively on the vertex set of $X$.
\end{enumerate}
In particular, the topological realization of the barycentric subdivision of $X$ is a topological model for $\F_2$, and hence $G$ is compactly presented.
\end{prop}

For a topological model $X$ of $G$ of type $\F_2$, by standard techniques we may add cells to kill higher homotopy, and get a contractible $G$-complex $X'$ on which $G$ acts simplicially with compact open stabilizers . Then the assumption on cell stabilizers  implies that the collection of $i$-cells of $X'$ is a proper $G$-set and hence $C_i(X',\Q)$ is a proper permutation $\QG$-module. Since $X'$ is contractible, the augmented chain complex $(C_\bullet (X',\Q),\partial_\bullet)$ is a projective resolution of $\Q$ in $\QGdis$ and, since $X'^{(2)} = X^{(2)}$ has finitely many orbits of cells, the chain complex is finitely generated in degrees $0$, $1$ and $2$. In particular compactly presented TDLC-groups have type $\FP_2$ in $\QGdis$.


\section{Weak $n$-dimensional isoperimetric inequality }\label{sec:weakiso}
\subsection{(Pseudo-)Norms on vector spaces}
Given a vector space $V$ over a subfield $\F$ of the complex numbers, a {\it pseudo-norm} on $V$ is a nonnegative-valued scalar function $\norm{\argu}\colon V \to \mathbb{R}_+$ with the following properties:
\begin{enumerate}[label={(N\arabic*)}]
	\item (Subadditivity) $\norm{u+v}\leq\norm{u}+\norm{v}$ for all $u,v\in V$;
	\item (Absolute Homogeneity) $\norm{\lambda\cdot v}=|\lambda|\norm{v}$, for all $\lambda\in \F$ and $v\in V$.
\end{enumerate}
A pseudo-norm $\norm{\argu}$ on a vector space $V$ is said to be a {\it norm} if it satisfies the following additional property:
\begin{enumerate}[label={(N3)}]
	\item (Point-separation) $\norm{v}=0, v\in V\Rightarrow v=0$.
\end{enumerate}
Let $f\colon (V,\norm{\argu}_V)\to (W,\norm{\argu}_W)$ be a linear function between pseudo-normed vector spaces. We say that $f$ is {\it bounded} if there exists a constant $C>0$ such that $\norm{f(v)}_W\leq C\norm{v}_V$ for all $v\in V$. In such a case, we write $\norm{\argu}_W\preceq^f\norm{\argu}_V$ when the constant $C$ is irrelevant. Two different norms $\norm{\argu}$ and $\norm{\argu}'$ on $V$
are said to be {\it equivalent}, $\norm{\argu}\sim\norm{\argu}'$, if $\norm{\argu}\preceq^{\textup{id}}\norm{\argu}'\preceq^{\textup{id}}\norm{\argu}$. From here on the relation $\preceq^{\textup{id}}$ will be denoted as $\preceq$.

\subsection{$\ell_1$-norm on permutation $\QG$-modules} Let $\Q[\Omega]$ be a permutation $\QG$-module. In particular, $\Q[\Omega]$ is a $\Q$-vector space with linear basis $\Omega$. Therefore, the nonnegative-valued function
\begin{equation}\label{eq:l1norm}
\norm{\argu}^{\Omega}_1\colon\Q[\Omega]\to\Q_+,\quad\text{s.t.}\quad \sum_{\omega\in\Omega}\alpha_{\omega}\omega\mapsto\sum_{\omega\in\Omega}\abs{\alpha_{\omega}},
\end{equation}
defines a norm on $\Q[\Omega]$. As usual, we shall refer to $\norm{\argu}^{\Omega}_1$ as the {\it $\ell_1$-norm} on $\Q[\Omega]$. Notice that $\norm{\argu}^{\Omega}_1$ is $G$-equivariant.
\begin{prop}\label{prop:bounded} Let $\phi\colon\Q[\Omega]\to\Q[\Omega']$ be a morphism of finitely generated permutation $\QG$-modules. Then $\norm{\argu}^{\Omega'}_1\preceq^{\phi} \norm{\argu}^{\Omega}_1.$
\end{prop}
\begin{proof} 
This is a consequence of the $G$-invariance of the $\ell_1$-norm and the fact that the modules are finitely generated. Indeed, the morphism $\phi$ is described by a finite matrix $A=(a_{i j})$ with entries in $\QG$. Consider the $\ell_1$-norm $\norm{\argu}_1$ on $\QG$ and let $C=\max \norm{a_{i j}}$. Then $\norm{\phi (x)}_1^{\Omega'} \leq C \norm{x}_1^\Omega$ for every $x\in \Q[\Omega]$.
\end{proof}

The above proposition will be used for discrete permutation modules over $\QG$.

\subsection{Filling pseudo-norms on discrete $\QG$-modules} Let $M$ be a finitely generated discrete $\QG$-module. Since $\QGdis$ has enough projectives, there exists a finitely generated proper permutation $\QG$-module $\Q[\Omega]$ mapping onto $M$, that is, $\Q[\Omega]\stackrel{\partial} {\twoheadrightarrow} M$ and $G$ acts on $\Omega$ with compact open stabilizers  and finitely many orbits. The {\it filling pseudo-norm $\norm{\argu}_\partial$ on $M$ induced by $\partial$} is defined as
\begin{equation}\label{eq:norm induced}
\norm{m}_\partial=\inf\{\norm{x}^{\Omega}_1\mid x\in \Q[\Omega], \partial(x)=m\}.
\end{equation}
One easily verifies that $\norm{\argu}_\partial$ is subadditive and absolutely homogeneous. 
Note that 
\begin{equation}\label{eq:norm_preceq} \norm{\argu}_\partial\preceq^\partial\norm{\argu}_1^\Omega.\end{equation} It is an observation that an $\ell_1$-norm on a finitely generated discrete permutation $G$-module $\Q[\Omega]$ is equivalent to a filling norm.

\begin{prop}\label{prop: bounded finitely generated} 
	Morphisms between finitely generated discrete $\QG$-modules are bounded with respect to filling pseudo-norms.
\end{prop}
\begin{proof}
Let $f: M \to N$ be a morphism of finitely generated discrete $\QG$-modules. Since $M$ and $N$ are both finitely generated in $\QGdis$, there exist morphisms $\Q[\Omega_1]\stackrel{\partial_1}{\twoheadrightarrow }M$ and $\Q[\Omega_2]\stackrel{\partial_2}{\twoheadrightarrow }N$ such that each $\Q[\Omega_i]$ is a finitely generated proper permutation module. 
By the universal property of $\Q[\Omega_1]$ as a projective object, there is $\phi\colon \Q[\Omega_1] \to \Q[\Omega_2]$ such that the following diagram commutes:
\begin{equation*} \label{eq:diag1}
\begin{tikzpicture}[>=angle 90]
\matrix(a)[matrix of math nodes,
row sep=2.5em, column sep=2.5em,
text height=1.5ex, text depth=0.25ex]
{\Q[\Omega_1] &\Q[\Omega_2] \\
	M&N\\};
\path[dashed, ->](a-1-1) edge node[above]{$\phi$} (a-1-2);
\path[ ->>](a-1-1) edge node[left]{$\partial_1$}(a-2-1);
\path[->>](a-1-2) edge node[right]{$\partial_2$}(a-2-2);
\path[ ->](a-2-1) edge node[below]{$f$} (a-2-2);
\end{tikzpicture}
\end{equation*}
For any $m\in M$ and any $\varepsilon > 0$, let $x_m\in\Q[\Omega_1]$ such that $\partial_1(x_m)=m$ and $\norm{x_m}^{\Omega_1}_1 \preceq^{\partial_1} \norm{m}_{\partial_1} + \varepsilon$. Since $f(m)=\partial_2(\phi(x_m))$, one has
\begin{equation*}\label{eq:equinorm1}
\begin{array}{r l l l}
\norm{f(m)}_{\partial_2} &\preceq^{\partial_2} &\norm{\phi(x_m)}^{\Omega_2}_1 & \text{by~\eqref{eq:norm_preceq},}\\
&\preceq^\phi & \norm{x_m}^{\Omega_1}_1 & \text{by Proposition~\ref{prop:bounded},}\\
&\preceq^{\partial_1} & \norm{m}_{\partial_1} + \varepsilon.
\end{array}
\end{equation*}
Since $\varepsilon$ is arbitrary, we deduce $\norm{\argu}_{\partial_2}\preceq^{f} \norm{\argu}_{\partial_1}$.
\end{proof}
By considering the identity function on a finitely generated discrete $\QG$-module $M$, the previous proposition implies:
\begin{cor}\label{cor:psudonorms} 
Let $G$ be a TDLC-group. Any two filling pseudo-norms on a finitely generated discrete $\QG$-module $M$ are equivalent.

 In particular, all the filling pseudo-norms on a finitely generated proper permutation $\QG$-module $\Q[\Omega]$ are equivalent to $\norm{\argu}_1^{\Omega}$, and therefore they are all norms.
\end{cor}
The former implies that each finitely generated discrete $\QG$-module $M$ admits a unique filling pseudo-norm up to equivalence. Therefore, by abuse of notation, we denote by $\norm{\argu}_M$ any filling pseudo-norm of $M$ and we refer to $\norm{\argu}_M$ as the {\it filling pseudo-norm of $M$}.

\subsection{Undistorted submodules}

Let $M$ be a discrete $\QG$-module with a norm $\norm{\argu}$ and let $N$ be a finitely generated discrete $\QG$-submodule of $M$. Then $N$ is said to be {\em undistorted with respect to $\norm{\argu}$} if the restriction of $\norm{\argu}$ to $N$ is equivalent to a filling norm on $N$. In the case that $M$ is finitely generated and $N$ is undistorted with respect to the filling norm $\norm{\argu}_M$ we shall simply say that $N$ is {\em undistorted} in $M$.

We note that in general it is not the case that finitely generated submodules of $M$ are undistorted; we refer the reader to Section~\ref{sec:proof2} for counter-examples.

\begin{prop}\label{prop:inclusionbounded} 
Let $G$ be a TDLC-group. The filling pseudo-norm $\norm{\argu}_P$ of a finitely generated projective discrete $\QG$-module $P$ is a norm. Moreover, if $P$ is a direct summand of a finitely generated proper permutation module $\Q[\Omega]$, then $P$ is undistorted in $\Q[\Omega]$.
\end{prop}
\begin{proof}
Let $\Q[\Omega]$ be a finitely generated proper permutation module such that $P$ is a direct summand of $\Q[\Omega]$; see Proposition~\ref{prop:proj}. Let $\iota\colon P\to\Q[\Omega]$ be the inclusion and let $\pi\colon\Q[\Omega]\twoheadrightarrow P$ be the projection such that $\pi\circ\iota=\textup{id}_P$. Proposition~\ref{prop: bounded finitely generated} implies $\norm{\argu}_1 ^\Omega\preceq^\iota \norm{\argu}_P$ and $\norm{\argu}_P\preceq^\pi\norm{\argu}_1 ^\Omega$ on $P$. The former inequality implies that $\norm{\argu}_P$ is a norm, and both of them imply that $\norm{\argu}_P\sim\norm{\argu}_1^\Omega$ on $P$. 
 \end{proof}

More generally, this argument shows that a direct summand of any finitely generated discrete $\QG$-module, with the filling norm, is undistorted.

We conclude the section with a technical result about bounded morphisms that will be used later and relies on the proof of the previous proposition.
\begin{prop}\label{Prop:normHG}
Let $G$ be a TDLC-group and $H$ a closed subgroup of $G$. Let $M$ be a finitely generated and projective $\QG$-module in $\QGdis$ with filling norm $\norm{\argu}_M$. 
Regard $M$ as a $\QH$-module via restriction, and suppose that $N$ is a finitely generated direct summand of $M$ in $\QHdis$. 
Then $N$ is an undistorted $\QH$-module of $M$ with respect to the norm $\norm{\argu}_M$.
\end{prop}
\begin{proof} 
The $\QH$-module $N$ is projective since the restriction of $M$ is projective and hence $N$ is a direct summand of a projective $\QH$-module.
 
By Proposition~\ref{prop:inclusionbounded}, $M$ can be assumed to be a finitely generated proper permutation $\QG$-module $\Q[\Omega]$.
Note that the restriction of $\Q[\Omega]$ is a proper permutation $\Q[H]$-module. 
 
Since $N$ is finitely generated, there exists an $H$-subset $\Sigma$ of $\Omega$ such that $\Sigma/H$ is finite and $N$ is a $\QH$-submodule of $\Q[\Sigma]$. 
Since $N$ and $\Q[\Sigma]$ are direct summands of $\Q[\Omega]$ as $\QH$-modules, it follows that $N$ is a direct summand of the finitely generated proper permutation $\QH$-module $\Q[\Sigma]$.

Proposition~\ref{prop:inclusionbounded} implies that the pseudo-norm $\norm{\argu}_N$ is a norm and $\norm{\argu}_N \sim \norm{\argu}_1^\Sigma$ on $N$. Since $\norm{\argu}_1^\Sigma= \norm{\argu}^\Omega_1$ on $\Q[\Sigma]$, it follows that $\norm{\argu}_N\sim\norm{\argu}_1^{\Omega}$ on the elements of $N$.
\end{proof}

\subsection{Weak $n$-dimensional linear isoperimetric inequality }\label{subsec:weaklin}
 Let $G$ be a TDLC-group of type $\FP_{n+1}$. Then there exists a partial proper permutation resolution 
\begin{equation} \label{res:defn}
\xymatrix{\Q[\Omega_{n+1}]\ar[r]^{\delta_{n+1}}&\Q[\Omega_n]\ar[r]^{\delta_n}&\cdots\ar[r]&\Q[\Omega_1]\ar[r]^{\delta_1}&\Q[\Omega_0]\ar[r]&\Q\ar[r]&0\\ }
\end{equation}
of finite type, i.e. it consists of finitely generated discrete $\QG$-modules. We say that $G$ satisfies the {\it weak $n$-dimensional linear isoperimetric inequality} if $\ker(\delta_n)$ is an undistorted submodule of $\Q[\Omega_n]$. The special case for $n=1$ is referred as the {\it weak linear isoperimetric inequality}.

Note that, by Proposition~\ref{prop: bounded finitely generated}, $\norm{\argu}^{\Omega_n}_1\preceq^\imath \norm{\argu}_{\ker(\partial_n)}$ where $\imath\colon \ker(\partial_n) \to \Q[\Omega_n]$ is the inclusion. Hence, the weak $n$-dimensional linear isoperimetric inequality is equivalent to the existence of a constant $C>0$ such that 
$\norm{\argu}_{\ker(\partial_n)} \leq C \norm{\argu}^{\Omega_n}_1$ on $\ker(\partial_n)$.

The proof of the following proposition is an adaption of the proof of~\cite[Theorem 3.5]{HM16} that we have included for the reader's convenience.
\begin{prop}\label{prop:AlgebraicDef} For a TDLC-group $G$ of type $\FP_{n+1}$, the property of satisfying the weak linear $n$-dimensional isoperimetric inequality is independent of the choice of the proper permutation resolution of finite type in $\QGdis$.
\end{prop}	
\begin{proof}
Let $(\Q[{\Omega}_i],{\partial}_{i}), (\Q[\Lambda_i],\delta_{i})$ be a pair of proper permutation resolutions of $\Q$ which contain finitely generated modules for degrees $i=0,\dots,n+1$. Suppose $G$ satisfies the weak $n$-dimensional linear isoperimetric inequality with respect to the resolution $(\Q[\Lambda_i],\delta_{i})$. Hence there is $C>0$ such that
\begin{equation}\label{eq:isopLamda} \norm{x}_{\ker(\delta_{n})} \leq C \norm{x}^{\Lambda_n}_1.\end{equation}
 for all $x\in\ker(\delta_{n})$.

Since any two projective resolutions of $\Q$ are chain homotopy equivalent, there exist chain maps $f\colon (\Q[{\Omega}_i],{\partial}_{ i})\to (\Q[\Lambda_i],\delta_{i}))$ and $g \colon (\Q[\Lambda_i],\delta_{i}) \rightarrow (\Q[{\Omega}_i],{\partial}_{ i})$, and a 1-differential $h\colon (\Q[{\Omega}_i],{\partial}_{i}) \rightarrow (\Q[{\Omega}_i],{\partial}_{i}) $ such that 
	\begin{equation}\label{eq:homotopy}
\partial_{i+1} \circ h_i + h_{i-1} \circ \partial_i = g_i \circ f_i - \mathsf{Id}.
	\end{equation}
Diagrammatically, one has
\begin{equation}\label{eq:diag2}
\xymatrix{
	\cdots\ar[r]&\Q[\Omega_{n+1}]\ar[r]_-{\partial_{n+1}}\ar@<-2pt>[d]_-{f_{n+1}}&\Q[{\Omega}_{n}]\ar[r]_-{\partial_{n}}\ar@<-2pt>[d]_-{f_n} \ar@/^-3ex/[l]_-{h_{n}}&\Q[\Omega_{n-1}]\ar@<-2pt>[d]_-{f_{n-1}}\ar@/^-3ex/[l]_-{h_{n-1}}\ar[r]&\cdots\\
	\cdots\ar[r]&\Q[\Lambda_{n+1}]\ar[r]_-{\delta_{n+1}}\ar@<-2pt>[u]_-{g_{n+1}}&\Q[{\Lambda}_{n}]\ar[r]_-{\delta_{n}}\ar@<-2pt>[u]_-{g_n}&\Q[\Lambda_{n-1}]\ar@<-2pt>[u]_-{g_{n-1}}\ar[r]&\cdots
	}
\end{equation}

Since $g_{n+1},f_n$ and $h_n$ are morphisms between finitely generated discrete $\QG$-modules, Proposition~\ref{prop: bounded finitely generated} applies and, therefore, the constant $C$ defined above can be assumed to satisfy:
\begin{enumerate}[label={(D\arabic*)}]
\item $\norm{g_{n+1}(\lambda)}^{\Omega_{n+1}}_1\leq C \norm{\lambda}^{\Lambda_{n+1}}_1$, for all $\lambda\in\Q[\Lambda_{n+1}]$;
\item $\norm{f_{n}(\omega)}^{\Lambda_{n}}_1\leq C \norm{\omega}^{\Omega_{n}}_1$, for all $\omega\in\Q[\Omega_{n}]$; and
\item $\norm{h_{n}(\omega)}^{\Omega_{n+1}}_1\leq C \norm{\omega}^{\Omega_{n}}_1$, for all $\omega\in\Q[\Omega_{n}]$.
\end{enumerate}

We prove below that that there is a constant $D>0$ such that for any $\alpha\in\ker(\partial_{n})$ and $\epsilon>0$
\[ \| \alpha \|_{\ker(\partial_{n})} \leq D \| \alpha \|^{\Omega_n}_1 + D\epsilon .
 \]
Then it follows that $G$ satisfies the weak $n$-dimensional linear isoperimetric inequality with respect to the resolution $(\Q[\Omega_i],\partial_{i})$ by letting $\epsilon \to 0$.

Let $\alpha\in \ker(\partial_{n})$ and $\epsilon>0$. By the diagram~\eqref{eq:diag2}, it follows that $f_n(\alpha)\in\ker(\delta_{n})=\delta_{n+1}(\Q[\Lambda_{n+1}]).$ Since $\Q[\Lambda_{n+1}]$ is finitely generated, we can consider the filling-norm $\norm{\argu}_{\ker(\delta_{n})}$ to be induced by $\delta_{n+1}$. Therefore, by the definition of the filling norm $\norm{\argu}_{\ker(\delta_{n})}$ there is $\beta\in\Q[\Lambda_{n+1}]$ such that $\delta_{n+1}(\beta)=f_n(\alpha)$ and
\begin{equation}\label{eq:epsilon}
\norm{\beta}^{\Lambda_{n+1}}_1\leq\norm{f_n(\alpha)}_{\ker(\delta_{n})}+\epsilon.
\end{equation}
By evaluating $\alpha$ in Equation~\ref{eq:homotopy}, we can write
\begin{align}\label{eq:eq1}
 \alpha &= g_n(f_n(\alpha))-\partial_{n+1}(h_n(\alpha)) \\ \nonumber
 &= g_n(\delta_{n+1}(\beta))-\partial_{n+1}(h_n(\alpha))\\ \nonumber
 &=\partial_{n+1}\left (g_{n+1} (\beta )- h_n ( \alpha ) \right). \nonumber
 \end{align}
 Hence

\[
 \begin{array}{rll}
 \| \alpha \|_{\ker(\partial_{n})} & \leq \left \| g_{n+1}(\beta) - h_n(\alpha) \right \|^{\Omega_{n+1}}_1 & \text{ by \eqref{eq:eq1} and definition of filling norm } \\[2pt] 
 & \leq\left \| g_{n+1}(\beta) \right \|^{\Omega_{n+1}}_1 + \left \| h_n(\alpha ) \right \|^{\Omega_{n+1}}_1 & \\[2pt] 
 & \leq C \| \beta \|^{\Lambda_{n+1}}_1 + C \| \alpha \|^{\Omega_n}_1 & \text{ by inequalities (D1) and (D3) } \\[2pt] 
 & \leq C \|f_n(\alpha) \|_{\ker(\delta_n)} + C \epsilon + C \| \alpha \|^{\Omega_n}_1 &\text{ by inequality \eqref{eq:epsilon} }\\[2pt] 
 & \leq C ^2 \| f_n(\alpha) \|^{\Lambda_n}_1 + C \| \alpha \|^{\Omega_n}_1 + C \epsilon & \text{ by inequality~\eqref{eq:isopLamda} }\\[2pt] 
 & \leq C ^3 \| \alpha \|^{\Omega_n}_1 + C \| \alpha \|^{\Omega_n}_1 + C \epsilon & \text{ by inequality (D2). } \qedhere \nonumber
 \end{array} 
\]

\subsection{Weak $0$-Dimensional Linear Isoperimetric Inequality and Profinite Groups}

As previously mentioned, a group is profinite if and only if it is a compact TDLC-group~\cite[Proposition 0]{Se02}. The following statement is a simple application of the definitions of this section. 
\begin{prop}
Let $G$ be a TDLC-group. Then $G$ is compact if and only if it is compactly generated and satisfies a weak $0$-dimensional linear isoperimetric inequality.
\end{prop}

The only if direction of the proposition is immediate. Indeed, if $G$ is a compact TDLC-group, then the trivial $G$-module $\Q$ is projective in $\QGdis$. In this case, one can read the weak $0$-dimensional isoperimetric inequality from the resolution $0\to\Q\to\Q\to0$. 

For the rest of the section, suppose that $G$ is a TDLC-group satisfying a weak $0$-dimensional linear isoperimetric inequality. Let $\Gamma$ be a Cayley-Abels graph of $G$, let $\dist$ be the combinatorial path metric on the set of vertices $V$ of $\Gamma$, and let $E$ denote the set of edges of $\Gamma$. In order to prove that $G$ is profinite, it is enough to show that $V$ is finite. 

Choose an orientation for each edge of $\Gamma$ and consider the augmented rational cellular chain complex of $\Gamma$,
\[ \Q[E] \stackrel{\delta}{\to} \Q[V] \stackrel{\varepsilon}{\to} \Q \to 0 .\]
Since $\Gamma$ is Cayley-Abels graph, this is a partial proper permutation resolution. 

Following ideas in~\cite{FMP18}, define a partial order $\preceq$ on $\Q[E]$ as follows. For $\nu, \mu \in \Q[E]$, $\nu=\sum_{e\in E} t_e e$ and $\mu=\sum_{e\in E} s_e e$, then $\nu\preceq \mu$ if and only if $t_e^2\leq t_e s_e $ for every $e\in E$. Observe that if $\nu \preceq \mu$ then
$\|\mu\|_1^E=\|\mu-\nu\|_1^E+\|\nu\|_1^E$; in particular $\|\nu\|_1^E \leq \| \mu \|_1^E$. An element $\nu \in \Q[E]$ is called integral if $t_e \in \Z$ for each $e$. Define analogously $\preceq$ on $\Q[V]$. 

\begin{lem}\label{lem:FlemmingMP}
Suppose that $\mu \in \Q[E]$ is integral and $\delta(\mu)=m(v-u)$ where $u,v\in V$ and $m$ is a positive integer. Then there is an integral element $\nu\in \Q[E]$ such that $\delta(\nu)=v-u$ and $\nu \preceq \mu$ and $\|\nu\|^E_1\geq \dist(u,v)$. 
\end{lem}
\begin{proof}[Sketch of the proof]
Suppose $\mu=\sum_{e\in E} s_e e$. Consider a directed multigraph $\Xi$ (multiple edges between distinct vertices are allowed) with vertex set $V$ and such that for each $e\in E$
if $s_e\geq 0$ then there are $|s_e|$ edges from $a$ to $b$ where $\delta(e)=b-a$; and if $s_e<0$ then there are $|s_e|$ from $b$ to $a$. The degree sum formula for directed graphs implies that $u$ and $v$ are in the same connected component of $\Xi$. It is an exercise to show that there is a directed path $\gamma$ from $u$ to $v$ in $\Xi$ that can be assumed to be injective on vertices. The path $\gamma$ induces an element $\nu\in \Q[E]$ such that if $\nu=\sum_{e\in E} t_e e$ then $t_e=\pm 1$ and $\nu \preceq \mu$. Moreover $\gamma$ induces a path in $\Gamma$ from $u$ to $v$ and hence $\|\nu\|_1^E\geq \dist(u,v)$.
\end{proof}

Suppose, for a contradiction, that $V$ is an infinite set. Fix $v_0\in V$. For every $n\in\N$, let $v_n\in V$ such that $\dist(v_0,v_n)\geq n$. Note that such a vertex $v_n$ always exists since $\Gamma$ is locally finite and connected. Let $\alpha_n = v_n-v_0$ and observe that $\alpha_n \in \ker(\varepsilon)$ and $\|\alpha_n\|_1^V = 2$. We will show that $\| \alpha_n \|_{\ker(\varepsilon)} \geq n$ for every $n$, and hence $G$ cannot satisfy a weak $0$-dimensional linear isoperimetric inequality. Fix $n\in \N$, and let $\mu=\sum_{e\in E} s_e e \in \Q[E]$ such that $\delta(\mu)=\alpha_n=v_n-v_0$.Then there is $m\in \N$ such that $m\mu$ is integral. Since $\delta(m\mu)=m(v_n-v_0)$, Lemma~\ref{lem:FlemmingMP} implies that there is $\nu_1 \in \Q[E]$ such that $\delta(\nu_1)=v_n-v_0$ and $\nu_1 \preceq m\mu$ and $\|\nu\|_1\geq \dist(v_0,v_n)$. Let $\mu_1=m\mu - \nu_1$ and note that $\mu_1$ is integral, $\delta(\mu_1) = (m-1)(v_n-v_0)$, and 
\[ \| m \mu \|_1^E = \| \mu_1 \|_1^E + \| \nu_1 \|_1^E \geq \| \mu_1 \|_1^E + \dist(v_0,v_n).\]
An induction argument on $m$ then proves that $\| m \mu \|_1^E \geq m \dist(v_0,v_n)$ and hence $\| \mu \|_1^E \geq \dist(v_0,v_n)$. Since $\mu$ was an arbitrary element such that $\delta(\mu)=\alpha_n$, it follows tha that $\|\alpha_n\|_{\kernel{\delta}} \geq \dist(v_0,v_n)\geq n$. 
\end{proof}
\section{Proof of Subgroup Theorem} \label{sec:proof1}

The proof of the theorem relies on the following lemma. Let $G$ be a TDLC-group of type $\FP_n$ and $H$ a closed subgroup of $G$ of type $\FP_n$.

\begin{lem}
\label{lem:res}
There are partial proper permutation resolutions $$\Q[\Omega_n] \xrightarrow{\delta_n} \Q[\Omega_{n-1}] \to \cdots \to \Q[\Omega_0] \to \Q \to 0,$$ $$\Q[\Sigma_n] \xrightarrow{\partial_n} \Q[\Sigma_{n-1}] \to \cdots \to \Q[\Sigma_0] \to \Q \to 0$$ of $\Q$ in $\QHdis$ and $\QGdis$ respectively, satisfying the following properties.
\begin{enumerate}
\item $\Omega_0,\ldots,\Omega_n$ are finitely generated $H$-sets;
\item $\Sigma_0,\ldots,\Sigma_n$ are finitely generated $G$-sets;
\item restricting the $G$-action on each $\Sigma_i$ to $H$, $\Omega_i$ is an $H$-subset of $\Sigma_i$ via $\iota_i: \Omega_i \to \Sigma_i$;
\item the diagram
\[\xymatrix{
\ker(\delta_n) \ar[r] \ar[d] & \Q[\Omega_n] \ar[r]^{\delta_n} \ar[d]^{\Q[\iota_n]} & \cdots \ar[r] & \Q[\Omega_0] \ar[r]^{\delta_0} \ar[d]^{\Q[\iota_0]} & \Q \ar[r] \ar@{=}[d] & 0 \\
\ker(\partial_n) \ar[r] & \Q[\Sigma_n] \ar[r]^{\partial_n} & \cdots \ar[r] & \Q[\Sigma_0] \ar[r]^{\partial_0} & \Q \ar[r] & 0
}\]
of $\Q[H]$-modules commutes;
\item $\coker(\ker(\delta_n) \to \ker(\partial_n))$ is a projective $\QH$-module.
\end{enumerate}
\end{lem}
\begin{proof}
Take a partial proper permutation resolution $$\Q[\Sigma_n] \xrightarrow{\partial_n} \Q[\Sigma_{n-1}] \to \cdots \to \Q[\Sigma_0] \to \Q \to 0$$ of $\Q$ in $\QGdis$. We construct the required resolution $$\Q[\Omega_n] \xrightarrow{\delta_n} \Q[\Omega_{n-1}] \to \cdots \to \Q[\Omega_0] \to \Q \to 0$$ in $\QHdis$ by induction on $n$. So suppose we have already constructed a diagram
\[\xymatrix{
\ker(\delta_{n-1}) \ar[r] \ar[d] & \Q[\Omega_{n-1}] \ar[r]^{\delta_{n-1}} \ar[d]^{\Q[\iota_{n-1}]} & \cdots \ar[r] & \Q[\Omega_0] \ar[r]^{\delta_0} \ar[d]^{\Q[\iota_0]} & \Q \ar[r] \ar@{=}[d] & 0 \\
\ker(\partial_{n-1}) \ar[r] & \Q[\Sigma_{n-1}] \ar[r]^{\partial_{n-1}} & \cdots \ar[r] & \Q[\Sigma_0] \ar[r]^{\partial_0} & \Q \ar[r] & 0
}\]
satisfying the conditions for $n-1$ (this is trivial for the base case $n=0$).

Write $\iota$ for the induced map $\ker(\delta_{n-1}) \to \ker(\partial_{n-1})$; by hypothesis, there is a map $\pi: \ker(\partial_{n-1}) \to \ker(\delta_{n-1})$ such that $\pi\iota$ is the identity on $\ker(\delta_{n-1})$. Since $H$ has type $\FP_n$, $\ker(\delta_{n-1})$ is finitely generated; pick a finite generating set $x_1,\ldots,x_k$ and pick a preimage $y_i$ of each element $x_i$ in $\Q[\Sigma_n]$, via the map $\Q[\Sigma_n] \xrightarrow{\partial_n} \ker(\partial_{n-1}) \xrightarrow{\pi} \ker(\delta_{n-1})$. Each $y_i$ is a finite sum $\sum_{j=1}^{j_i} a_{ij}\alpha_{ij}$ with $\alpha_{ij} \in \Sigma_n$ and $a_{ij} \in \Q$. Now let $\Omega_n$ be the (finitely generated) $H$-subset of $\Sigma_n$ generated by the $\alpha_{ij}$. We get an induced map $\pi \partial_n \Q[\iota_n]: \Q[\Omega_n] \to \ker(\delta_{n-1})$ extending the commutative diagram as required; it only remains to check condition 5.

To see this, consider the following commutative diagram in $\QHdis$
\begin{equation}\nonumber
\xymatrix{
& 0 \ar[d] & 0 \ar[d] & 0 \ar[d] \\
0 \ar[r] & \ker(\delta_n) \ar[d]^{\iota'} \ar[r] & \Q[\Omega_n] \ar[d]^{\Q[\iota_n]} \ar[r]^{\delta_n} & \ker(\delta_{n-1}) \ar[r] \ar[d]^{\iota} & 0 \\
0 \ar[r] & \ker(\partial_n) \ar[r] \ar[d] & \Q[\Sigma_n] \ar[r]^{\partial_n} \ar[d] & \ker(\partial_{n-1}) \ar[r] \ar[d] & 0 \\
0 \ar[r] & \coker(\iota') \ar[r] \ar[d] & \coker(\Q[\iota_n]) \ar[r] \ar[d] & \coker(\iota) \ar[r] \ar[d] & 0 \\
& 0 & 0 & 0.}
\end{equation}
Note that the diagram consists of exact rows and exact columns. Since $\Q[\Omega^H_n]$ is a direct summand of $\Q[\Sigma^G_n]$ in $\QHdis$, it follows that each $\coker(\Q[\iota_n])$ is projective; $\coker(\iota)$ is projective by hypothesis. Then exactness of the bottom row implies that $\coker(\iota')$ is projective.
\end{proof}

\begin{rem}\label{rem:openvsclosed}
In $\QGdis$, it is possible to develop a homological mapping cylinder
argument analogous to~\cite[Proposition 4.1]{AMP19} that yields a similar conclusion to Lemma \ref{lem:res} but only for open subgroups of $G$. This argument was developed in a preliminary version of this article.
\end{rem}

\begin{proof}[Proof of Theorem~\ref{thm:mainineq1}]
Since $G$ and $H$ have type $\FP_n$, we may use the partial proper permutation resolutions described in Lemma \ref{lem:res}; we keep the notation from there. Because $G$ has type $\FP_{n+1}$ and $\ccd_\Q(G) = n+1$, $\ker(\partial_n)$ is finitely generated (in $\QGdis$) and projective; because $H$ has type $\FP_{n+1}$ and $\coker(\iota')$ is projective, $\ker(\delta_n)$ is a finitely generated (in $\QHdis$) summand of $\ker(\partial_n)$. So:
\begin{enumerate}
\item $\|\argu\|_{\ker(\delta^H_n)}\sim\|\argu\|_{\ker(\partial^G_n)}$
on the elements of $\ker(\delta^H_n)$, by Proposition \ref{Prop:normHG};
\item $\|\argu\|^{\Omega^H_n}_1\sim\|\argu\|^{\Sigma^G_n}_1$ on the elements of $\Q[\Omega^H_n]$, because $\Omega_n$ is a subset of $\Sigma_n$;
\item $\|\argu\|_{\ker(\partial^G_n)}\sim \|\argu\|^{\Sigma_n^G}_1$ on the elements of $\ker(\partial^G_n)$, because $G$ satisfies the weak $n$-dimensional linear isoperimetric inequality.
\end{enumerate}
Therefore $ \|\argu\|_{\ker(\delta^H_n)}\sim \|\argu\|^{\Omega^H_n}_1$ on the elements of $\ker(\delta^H_n)$, i.e. $H$ satisfies the weak $n$-dimensional isoperimetric inequality.
\end{proof}

\section{Weak linear isoperimetric inequality and hyperbolicity} \label{sec:proof2}

The notion of linear isoperimetric inequality was used to characterise discrete hyperbolic groups by Gersten \cite{Ge96a}. Different generalizations of Gersten's result have been presented by various authors; see for example~\cite{GrMa08},~\cite{Mi02} and \cite{HM16}. In particular, Manning and Groves \cite{GrMa08} reformulated Gersten's argument to provide a homological characterization of simply connected hyperbolic $2$-complexes by means of a {\it homological isoperimetric inequality}. Here we use results from~\cite{GrMa08} to provide an analogue characterization of hyperbolic TDLC-groups. 

Let $X$ be a complex with $i$-skeleton denoted by $X^{(i)}$. Consider the cellular chain complex $(C_\bullet (X,\Q),\partial_\bullet)$ of $X$ with rational coefficients. 
Each vector space $C_i(X,\Q)$ is $\Q$-spanned by the collection of $i$-cells $\sigma$ of $X$. An $i$-chain $\alpha$ is a formal linear combination $\sum_{\sigma\in X^{(i)}} r_\sigma\sigma$
where $r_\sigma \in \Q$. The $\ell_1$-norm on $C_i(X,\Q)$ is defined as 
\[ \|\alpha\|^{X,i}_1= \sum |r_\sigma |,\]
where $|\argu|$ denotes the absolute value function on $\Q$.

\begin{defn}[ \protect{\cite[Definition 2.18]{GrMa08}} Combinatorial path]
Let $X$ be a complex. Suppose $I$ is an interval with a cellular structure. A {\em combinatorial path} $I \to X^{(1)}$ is a cellular path sending $1$-cells to either $1$-cells or $0$-cells. A {\em combinatorial loop} is a combinatorial path with equal endpoints. 
\end{defn}

From here on, to simplify notation, the $1$-chain induced by a combinatorial loop $c$ in $X$ is denoted by $c$ as well. 

\begin{defn}[\protect{\cite[Definition 2.28]{GrMa08}} Linear Homological isoperimetric inequality] \label{defn:manningisoperimetric}
Let $X$ be a simply connected complex. We say that $X$ satisfies the {\em linear homological isoperimetric inequality} if there is a constant $K\geq 0$ such that for any combinatorial loop $c$ in $X$ there is some $\sigma \in C_2(X, \Q)$ with $\partial(\sigma)=c$ satisfying 
\begin{equation}\|\sigma\|^{X,2}_1\leq K \|c\|^{X,1}_1.\end{equation}
\end{defn}

\begin{defn}
Let $G$ be a compactly presented TDLC-group. There exists a simply connected $G$-complex $X$ with compact open cell stabilizers, the $2$-skeleton $X^{(2)}$ is compact modulo $G$, the $G$-action is cellular and an element in $G$ fixing a cell setwise fixes it already pointwise. The group $G$ satisfies the \emph{linear homological isoperimetric inequality} if $X$ does. 
\end{defn}

The above definition is independent of the choice of $X$ as a consequence of Proposition~\ref{prop:AlgebraicDef}, the fact that a compactly presented TDLC-group has type $\FP_2$, and the following statement.

\begin{prop}\label{prop:hom=iso}
Suppose $G$ is a compactly presented TDLC-group and $X$ is a topological model of $G$ of type $\F_2$. Then $G$ satisfies the weak linear isoperimetric inequality if and only if $X$ satisfies the linear homological isoperimetric inequality. 
\end{prop}
\begin{proof}
The augmented cellular chain complex $(C_\bullet(X, \Q), \partial_\bullet)$ of $X$ is a proper partial permutation resolution of $\Q$ of type $\FP_2$. The module $C_i(X,\Q)$ is a proper permutation module and we can take as its filling norm $\norm{\argu}_{C_i}$ the $\ell_1$-norm induced by $G$-set of $i$-cells.

The weak linear isoperimetric inequality means that the filling norm $\norm{\argu}_{Z_1}$ of $Z_1(X,\Q)$ is equivalent to the restriction of $\norm{\argu}_{C_1}$ to $Z_1(X,\Q)$. Hence there is a constant $C>0$ such that $\norm{\argu}_{Z_1} \leq C\norm{\argu}_{C_1}$ on $Z_1(X,\Q)$. To prove the linear homological isoperimetric inequality is enough consider non-trivial combinatorial loops, the inequality is trivial otherwise. Let $c$ be a non-trivial combinatorial loop and let $\mu \in C_2(X)$ such that $\partial \mu = c$ and $\norm{\mu}_{C_2} \leq \norm{c}_{Z_1}+1$. In particular, $\norm{\mu}_{C_2} \leq \norm{c}_{Z_1}+\norm{c}_{C_1}$, since $\norm{c}_{C_1}$ is a positive integer. It follows that $\norm{\mu}_{C_2} \leq (C+1)\norm{c}_{C_1}$ for any non-trivial combinatorial loop.

Conversely, suppose that $X$ satisfies the linear homological isoperimetric inequality for a constant $C$.
Let $\gamma \in Z_1(X,\Q)$.  Then the filling norm on $\gamma \in Z_1(X, \Q)$ is given by 
$\norm{\gamma}_{Z_1} = \inf \{ \norm{\mu}_{C_2} \colon \mu\in C_2(X,\Q), \partial \mu =\gamma \}$.  There is an integer $m$ such that $m\gamma$ is an integer cycle.  Then $m\gamma= c_1+c_2+\cdots +c_k$ where each $c_i$ is a cycle induced by a combinatorial loop, and $\norm{m\gamma}_{C_1}=\sum \norm{c_i}_{C_1}$, see~\cite[Lemma A.2]{Ge98}. Then there are $2$-cycles $\sigma_i \in C_2(X,\Q)$ such that $\partial \sigma_i = c_i$ and $\norm{\sigma_i}_{C_2} \leq C \norm{c_i}_{C_1}$.  It follows that 
\[ \norm{m\gamma}_{Z_1} \leq  \norm{ \sum_i \sigma_i}_{C_2} \leq  \sum_i \norm{ \sigma_i}_{C_2} \leq  C \sum_i  \norm{c_i}_{C_1}  = C \norm{m\gamma}_{C_1}. \]
Since both $\norm{\argu}_{Z_1}$ and $\norm{\argu}_{C_1}$ are homogeneous (see (N2) in Section~\ref{sec:weakiso}), the previous inequality implies that 
  $\norm{\argu}_{Z_1} \leq C \norm{\argu}_{C_1}$ on $Z_1(X,\Q)$. On the other hand, since the inclusion $Z_1(X,\Q) \hookrightarrow C_1(X,\Q)$ is bounded, there is another constant $C'$ such that $\norm{\argu}_{C_1} \leq C' \norm{\argu}_{Z_1}$ on $Z_1(X,\Q)$. Therefore the norms $\norm{\argu}_{Z_1}$ and $\norm{\argu}_{C_1}$ are equivalent on $Z_1(X, \Q)$.
\end{proof}

Below we recall a characterization of hyperbolic simply connected 2-complexes from~\cite{GrMa08}.
 
\begin{prop}\cite[Proposition 2.23, Lemma 2.29, Theorem 2.30 ]{GrMa08}\label{thm:manning1} \label{thm:manning2}
Let $X$ be a simply connected $2$-complex. 
\begin{enumerate}
\item If $X^{(1)}$ is hyperbolic, then $X$ satisfies the linear homological isoperimetric inequality.
\item If there is a constant $M$ such that the attaching map for each $2$-cell in $X$ has length at most $M$, and $X$ satisfies a linear homological isoperimetric inequality; then $X^{(1)}$ is hyperbolic.
\end{enumerate}
\end{prop}

\begin{proof}[Proof of Theorem~\ref{thm:characterization}] Let $G$ be a compactly generated TDLC-group. Suppose that $G$ is hyperbolic. By Proposition~\ref{prop:F2model}, $G$ is compactly presented and there is a topological model $X$ of $G$ of type $\textup{F}_2$. By Proposition~\ref{prop:svarc}, the $1$-dimensional complex $X^{(1)}$ is quasi-isometric to a Cayley-Abels graph of $G$. It follows that $X^{(1)}$ is hyperbolic. Hence, Propositions~\ref{prop:hom=iso} and~\ref{thm:manning1} imply that $G$ satisfies the weak linear isoperimetric inequality. 
 
Conversely, suppose that $G$ is compactly presented and satisfies the weak linear isoperimetric inequality. 
Proposition~\ref{prop:F2model0} implies that there is a topological model $X$ of $G$ of type $\textup{F}_2$.
By Proposition~\ref{prop:hom=iso}, $X$ satisfies the linear homological isoperimetric inequality.
Since the $G$-action on the $2$-skeleton $X^{(2)}$ has finitely many $G$-orbits of $2$-cells, there is a constant $M$ such that the attaching map for each $2$-cell in $X$ has length at most $M$. 
Then Proposition~\ref{thm:manning2} implies that $X^{(1)}$ is hyperbolic. By Proposition~\ref{prop:svarc}, the Cayley-Abels graphs of $G$ are hyperbolic.
\end{proof}

\section{Small cancellation quotients of amalgamated products of profinite groups}  \label{sec:SmallCancellation}

This section relies on small cancellation theory over free products with amalgamation as developed in Lyndon-Schupp textbook~\cite[Chapter V, Section 11]{LySc01}. Before we state the main result of the section, we recall some of the terminology.

Let $A$ and $B$ be groups, and let $C$ be a common subgroup. A {\em reduced word}\footnote{In~\cite[Section V Chapter 11]{LySc01} this is also called normal form.} is a sequence  $x_1 \ldots x_n,n\geq0$, of elements of $A\ast_C B$ such that
\begin{enumerate}
	\item Each $x_i$ belongs to one of the factors $A$ or $B$.
	\item Successive $x_i,x_{i+1}$ belong to different factors.
	\item If $n>1$, no $x_i$ belongs to $C$.
	\item If $n=1$, then $x_1\neq 1$.
\end{enumerate}
A sequence $x_1 \ldots x_n$ is {\em semi-reduced} if it satisfies all the above items with $(2)$ replaced by 
\begin{itemize}
\item[(2')] The product of successive $x_i,x_{i+1}$ does not belong to $C$. 
\end{itemize}
Every element of $A\ast_C B$  can be represented as the product of the elements in a reduced word. Moreover, if $x_1\ldots x_n, n\geq 1$ is a reduced word, then the product $x_1 \cdots x_n$ is not trivial in $A\ast_CB$ (see~\cite[Theorem~2.6]{LySc01}). A reduced word $w = x_1\ldots x_n$ is {\em cyclically reduced} if $n=1$ or if $x_n$ and $x_1$ are in different factors of $A\ast_C B$. The word $w$ is {\em weakly cyclically reduced} if $n=1$ or if $x_nx_1\notin C$. 
\par A subset $R$ of words in $A\ast_C B$ is {\em symmetrized} if $r\in R$ implies $r$ is weakly cyclically reduced and every weakly cyclically reduced conjugate of $r^{\pm 1}$ is also in $R$. A symmetrized subset $R$ is {\em finite} if all elements represented by words in $R$ belong to a finite number of conjugacy classes in $A \ast_CB$. Let $R$ be a symmetrized subset of $A\ast_C B$. A word b is said to be a {\em piece} (relative to $R$) if there exist distinct elements $r_1$ and $r_2$ of $R$ such that $r_1=bc_1$ and $r_2 = bc_2$ in semi-reduced form.
\begin{itemize}
	\item[$C'(\lambda)$:] If $r\in R$ has semi-reduced form $r = bc$ where $b$ is a piece, then $|b| < \lambda |r|$. Further, $|r| > 1/\lambda$ for all $r\in R$.
\end{itemize}

\begin{thm}
Let $A \ast_C B$ be the amalgamated free product of the profinite groups $A,B$ over a common open subgroup $C$. Let $R$ be a finite symmetrized subset of $A\ast_C B$ that satisfies the $C'(1/12)$ small cancellation condition. Then the quotient $G= (A \ast_C B)/\nclose{R}$ is a hyperbolic TDLC-group with  $\ccd_\Q (G) \leq 2$.
 \end{thm}
 
Since the rational discrete cohomological dimension of a TDLC-group $G$ is less or equal to the geometric dimension of a contractible $G$-CW-complex acted on by $G$ with compact open stabilizers (see \cite[Fact~2.7]{CCC} for example), in order to prove the theorem, we construct a contractible cellular $2$-dimensional $G$-complex $X$ with compact open cell stabilizers, and such that its $1$-skeleton is hyperbolic. The $1$-skeleton is obtained as a quotient of the Bass-Serre tree of $A \ast_CB$, and then $X$ is obtained by pasting $G$-orbits of $2$-cells in one to one correspondence with conjugacy classes defined by $R$.

Recall that there exists a unique group topology on $A\ast_C B$ with the following properties (see~\cite[Proposition~8.B.9]{CoHa16} for example): the natural homomorphisms $A\to A\ast_C B$, $B\to A\ast_C B$, and $C\to A\ast_C B$  are topological isomorphisms onto open subgroups of $A\ast_C B$. Moreover, $A\ast_C B$ is a TDLC-group, and in particular, all open subgroups of $C$ form a local basis at the identity of compact open subgroups of $A\ast_C B$. 

\begin{proof}

Let $T$ be the Bass-Serre tree of $A\ast_C B$. Observe that $T$ is locally finite, the action of $A\ast_C B$  on $T$ is cobounded and it has compact open stabilizers. In particular, the action of $A\ast_C B$ on $T$ is geometric and hence $A\ast_C B$ is a hyperbolic TDLC-group. 

Let $N$ denote the normal closure of $R$ in $A\ast_CB$. By Greendlinger's lemma~\cite[Ch. V. Theorem 11.2]{LySc01}, the natural morphisms $A\to G$ and $B \to G$ are monomorphisms. It follows that the subgroup $N$ does not intersect $A$, $B$ and $C$. Thus the action of $N$ on the tree $T$ is free, and as a consequence $N$ is discrete. Hence $N$ is closed in $A\ast_C B$, and $G$ is a TDLC-group.

Let $X^{(1)}$ denote the quotient graph $T/N$. Since $N$ was acting freely and cellularly on $T$, the quotient map $\rho\colon T \to X^{(1)}$ is a covering map. Since $T$ is locally finite, $X^{(1)}$ is locally finite. Since the quotient map $A \ast_CB \to G$ is open, the action of $A\ast_C B$ on $T$ induces an action of $G$ on $X^{(1)}$ which is cobounded and has compact open stabilizers.
 
 Let $x_0$ be a fixed vertex of $T$ that we consider as the base point from now on. Since $T$ is simply connected, there is a natural isomorphism from $N$ to the fundamental group $\pi_1(X^{(1)}, \rho(x_0))$. Specifically, for each $g\in N$, let $\alpha_g$ be the unique path in $T$ from $x_0$ to $g.x_0$. Let $\gamma_g=\rho\circ\alpha_g$ be the closed path in $X^{(1)}$ induced by $\alpha_g$ based at $\rho(x_0)$. Thus, the isomorphism from $N$ to  $\pi_1(X^{(1)}, \rho(x_0))$ is defined by $g\mapsto \gamma_g$.

We are ready to define $X$. For $g\in G$ and $h\in N$,  let $g.\gamma_h$ be the translated closed path without an initial point, i.e., these are cellular maps from $S^1 \to X$. Consider the $G$-set $\Omega=\{g.\gamma_{r}\mid r\in R, g\in G\}$ of closed paths in $X^{(1)}$. Let $X$ be the $G$-complex obtained by attaching a $2$-cell to $X^{(1)}$ for every closed path in $\Omega$. In particular, the pointwise $G$-stabilizer  of a $2$-cell of $X$ coincides with the pointwise $G$-stabilizer of its boundary path and, therefore, compact and open. Then $X$ is a discrete $G$-complex of dimension $2$. Observe that the natural isomorphism from $N$ to  $\pi_1(X^{(1)}, \rho(x_0))$ implies that $X$ is simply connected. Moreover, since $R$ is finite, $X$ is a cobounded $2$-dimensional discrete $G$-complex.

We observe that $X$ is a $C'(1/6)$ complex and in particular the one-skeleton of $X$ is a Gromov hyperbolic graph with respect to the path metric, this is a well known consequence, see~\cite{MR1074477}. Let $R_1$ and $R_2$ be a pair of distinct 2-cells in $X$ such that the intersection of their boundaries contains an embedded path $\gamma$. We can assume, by translating by an element of $G$, that the base point of $X$ is either the initial vertex of $\gamma$ or the second vertex of $\gamma$. Let $\gamma'$ be the sub-path of $\gamma$ with initial point the base point of $X$; observe that $|\gamma| \leq |\gamma'| +1 $. Consider the boundary paths of $R_1$ and $R_2$ starting at the base point and oriented such that $\gamma'$ is an initial sub-path of both of them.
Consider the lifts of the chosen boundary paths of $R_1$ and $R_2$ in the tree $T$ starting from the base point. They intersect along the lifting of the path $\gamma'$.  Let us call this path $\hat{\gamma}$. Then the reduced word in $A \ast_C B$ corresponding to $\hat{\gamma}$ is a piece, and hence its length is bounded by $\frac{1}{12} |\partial R_i|$, for $i = 1,2$. We have the following inequality $$|\gamma| \leq |\gamma'| + 1 = |\hat{\gamma}| + 1 \leq \frac{1}{12}|\partial R_i| + \frac{1}{12}|\partial R_i|, \text{ for } i=1,2.$$ Hence $X$ is a $C'(1/6)$ complex.

 We conclude that the complex $X$ is contractible by using a well known argument of Ol\cprime shanski\u{\i}~\cite{MR1191619}. By a remark of Gersten~\cite[Remark 3.2]{MR919828}, if every spherical diagram in $X$ is reducible, then $X$ has trivial second homotopy group, and therefore $X$ is contractible because it is simply connected. Let $S \to X$ be a spherical diagram, and suppose that it is not reducible. Consider the dual graph $\Phi$ to the cellular structure of $S$, specifically $\Phi$ is the graph whose vertices are the two cells of $X$ and there is an edge between two vertices for each connected component of the intersection of the boundaries of the corresponding $2$-cells. Observe that $\Phi$ is planar. Since $X$ is a $C'(1/6)$ complex, the boundary paths of $2$-cells  are embedded paths and the intersection of the boundaries of any pair of $2$-cells is connected, and hence $\Phi$ is simplicial. Also since $X$ is $C'(1/6)$, every vertex of $\Phi$ has degree at least $6$. Since a finite planar simplicial graph has at least one vertex of degree  at most $5$, we have reached a contradiction and therefore the diagram $S \to X$ has to be reducible. The above sketched argument can be found in~\cite[Proof of Theorem 6.3]{Ma17} in a different framework.
 \end{proof}

\bibliographystyle{plain}

\begin{thebibliography}{10}

\bibitem{Ab72}
Herbert Abels.
\newblock Kompakt definierbare topologische {G}ruppen.
\newblock {\em Math. Ann.}, 197:221--233, 1972.

\bibitem{ABDY13}
Aaron Abrams, Noel Brady, Pallavi Dani, and Robert Young.
\newblock Homological and homotopical {D}ehn functions are different.
\newblock {\em Proc. Natl. Acad. Sci. USA}, 110(48):19206--19212, 2013.

\bibitem{AMP19}
Shivam Arora and Eduardo Mart\'inez-Pedroza.
\newblock Subgroups of word hyperbolic groups in rational dimension 2.
\newblock {\em Groups Geom. Dyn.}, 2020.

\bibitem{Bau}
Udo Baumgartner.
\newblock Totally disconnected, locally compact groups as geometric objects.
\newblock In {\em Geometric group theory}, pages 1--20. Springer, 2007.

\bibitem{BMW12}
Udo Baumgartner, R{\"o}gnvaldur~G M{\"o}ller, and George~A Willis.
\newblock Hyperbolic groups have flat-rank at most 1.
\newblock {\em Israel Journal of Mathematics}, 190(1):365--388, 2012.

\bibitem{BSW08}
Udo Baumgartner, G{\"u}nter Schlichting, and George~A Willis.
\newblock Geometric characterization of flat groups of automorphisms.
\newblock {\em Groups Geom. Dyn.}, 4(1):1--13, 2010.

\bibitem{BM91}
Mladen Bestvina and Geoffrey Mess.
\newblock The boundary of negatively curved groups.
\newblock {\em J. Amer. Math. Soc.}, 4(3):469--481, 1991.

\bibitem{Bo97}
Marc~Bourdon.
\newblock Immeubles hyperboliques, dimension conforme et rigidit\'{e} de
  {M}ostow.
\newblock {\em Geom. Funct. Anal.}, 7(2):245--268, 1997.

\bibitem{Br99}
Noel Brady.
\newblock Branched coverings of cubical complexes and subgroups of hyperbolic
  groups.
\newblock {\em J. London Math. Soc. (2)}, 60(2):461--480, 1999.

\bibitem{BH99}
Martin~R. Bridson and Andr\'{e} Haefliger.
\newblock {\em Metric spaces of non-positive curvature}, volume 319 of {\em
  Grundlehren der Mathematischen Wissenschaften [Fundamental Principles of
  Mathematical Sciences]}.
\newblock Springer-Verlag, Berlin, 1999.

\bibitem{BuMo97}
Marc Burger and Shahar Mozes.
\newblock Finitely presented simple groups and products of trees.
\newblock {\em C. R. Acad. Sci. Paris S\'{e}r. I Math.}, 324(7):747--752, 1997.

\bibitem{BL90}
Serei~V. Buyalo and Nina~D. Lebedeva.
\newblock Dimensions of locally and asymptotically self-similar spaces.
\newblock {\em Algebra i Analiz}, 19(1):60--92, 2007.

\bibitem{Cameron96}
Peter~J. Cameron.
\newblock Metric and topological aspects of the symmetric group of countable
  degree.
\newblock {\em European Journal of Combinatorics}, 17(2-3):135--142, 1996.

\bibitem{CCMT15}
 Pierre-Emmanuel Caprace, Yves~de~Cornulier, Nicolas~Monod, and Romain~Tessera.
\newblock Amenable hyperbolic groups.
\newblock {\em J. Eur. Math. Soc. (JEMS)},
  017(11):2903--2947, 2015.

\bibitem{Ca18}
Ilaria Castellano.
\newblock Rational discrete first degree cohomology for totally disconnected
  locally compact groups.
\newblock {\em Math. Proc. Cambridge Philos. Soc.}, 168(2):361--377, 2020.

\bibitem{CCC}
Ilaria Castellano and G~Corob Cook.
\newblock Finiteness properties of totally disconnected locally compact groups.
\newblock {\em Journal of Algebra}, 543:54--97, 2020.

\bibitem{CaWe16}
Ilaria Castellano and Thomas~S. Weigel.
\newblock Rational discrete cohomology for totally disconnected locally compact
  groups.
\newblock {\em J. Algebra}, 453:101--159, 2016.

\bibitem{CoHa16}
Yves Cornulier and Pierre de~la Harpe.
\newblock {\em Metric geometry of locally compact groups}, volume~25 of {\em
  EMS Tracts in Mathematics}.
\newblock European Mathematical Society (EMS), Z\"{u}rich, 2016.
\newblock Winner of the 2016 EMS Monograph Award.

\bibitem{FMP18}
Joshua~W. Fleming and Eduardo Mart\'{\i}nez-Pedroza.
\newblock Finiteness of homological filling functions.
\newblock {\em Involve}, 11(4):569--583, 2018.

\bibitem{MR919828}
Steve~M. Gersten.
\newblock Reducible diagrams and equations over groups.
\newblock In {\em Essays in group theory}, volume~8 of {\em Math. Sci. Res.
  Inst. Publ.}, pages 15--73. Springer, New York, 1987.

\bibitem{Ge96a}
Steve~M. Gersten.
\newblock A cohomological characterization of hyperbolic groups.
\newblock 1996.

\bibitem{Ge96}
Steve~M. Gersten.
\newblock Subgroups of word hyperbolic groups in dimension {$2$}.
\newblock {\em J. London Math. Soc. (2)}, 54(2):261--283, 1996.

\bibitem{Ge98}
Steve~M. Gersten.
\newblock Cohomological lower bounds for isoperimetric functions on groups.
\newblock {\em Topology}, 37(5):1031--1072, 1998.

\bibitem{MR1074477}
Steve~M. Gersten and H.~B. Short.
\newblock Small cancellation theory and automatic groups.
\newblock {\em Invent. Math.}, 102(2):305--334, 1990.

\bibitem{GrMa08}
Daniel Groves and Jason~F. Manning.
\newblock Dehn filling in relatively hyperbolic groups.
\newblock {\em Israel J. Math.}, 168:317--429, 2008.

\bibitem{HM14}
Richard~G. Hanlon and Eduardo Mart\'{\i}nez-Pedroza.
\newblock Lifting group actions, equivariant towers and subgroups of
  non-positively curved groups.
\newblock {\em Algebr. Geom. Topol.}, 14(5):2783--2808, 2014.

\bibitem{HM16}
Richard~G. Hanlon and Eduardo Mart\'{i}nez~Pedroza.
\newblock A subgroup theorem for homological filling functions.
\newblock {\em Groups Geom. Dyn.}, 10(3):867--883, 2016.

\bibitem{HOP14}
Sebastian Hensel, Damian Osajda, and Piotr Przytycki.
\newblock Realisation and dismantlability.
\newblock {\em Geom. Topol.}, 18(4):2079--2126, 2014.

\bibitem{KB02}
Ilya Kapovich and Nadia Benakli.
\newblock Boundaries of hyperbolic groups.
\newblock In {\em Combinatorial and geometric group theory ({N}ew {Y}ork,
  2000/{H}oboken, {NJ}, 2001)}, volume 296 of {\em Contemp. Math.}, pages
  39--93. Amer. Math. Soc., Providence, RI, 2002.

\bibitem{KrMo08}
Bernhard Kr\"{o}n and R\"{o}gnvaldur~G. M\"{o}ller.
\newblock Analogues of {C}ayley graphs for topological groups.
\newblock {\em Math. Z.}, 258(3):637--675, 2008.

\bibitem{LySc01}
Roger~C. Lyndon and Paul~E. Schupp.
\newblock {\em Combinatorial group theory}.
\newblock Classics in Mathematics. Springer-Verlag, Berlin, 2001.
\newblock Reprint of the 1977 edition.

\bibitem{Ma17}
Eduardo Mart\'{i}nez-Pedroza.
\newblock Subgroups of relatively hyperbolic groups of {B}redon cohomological
  dimension 2.
\newblock {\em J. Group Theory}, 20(6):1031--1060, 2017.

\bibitem{MS02}
David Meintrup and Thomas Schick.
\newblock A model for the universal space for proper actions of a hyperbolic
  group.
\newblock {\em New York J. Math.}, 8:1--7 (electronic), 2002.

\bibitem{Mi02}
Igor Mineyev.
\newblock Bounded cohomology characterizes hyperbolic groups.
\newblock {\em Q. J. Math.}, 53(1):59--73, 2002.

\bibitem{Mo02}
R{\"o}gnvaldur~G M{\"o}ller.
\newblock Structure theory of totally disconnected locally compact groups via
  graphs and permutations.
\newblock {\em Canadian Journal of Mathematics}, 54(4):795--827, 2002.

\bibitem{MR1191619}
A.~Yu. Ol\cprime~shanski\u{\i}.
\newblock {\em Geometry of defining relations in groups}, volume~70 of {\em
  Mathematics and its Applications (Soviet Series)}.
\newblock Kluwer Academic Publishers Group, Dordrecht, 1991.
\newblock Translated from the 1989 Russian original by Yu. A. Bakhturin.

\bibitem{Se02}
Jean-Pierre Serre.
\newblock Galois cohomology, corrected reprint of the 1997 english edition.
\newblock {\em Springer Monographs in Mathematics, Springer-Verlag, Berlin},
  pages 94720--3840, 2002.

\bibitem{Sw98}
Jacek \'{S}wi\polhk atkowski.
\newblock Trivalent polygonal complexes of nonpositive curvature and {P}latonic
  symmetry.
\newblock {\em Geom. Dedicata}, 70(1):87--110, 1998.

\bibitem{Vd36}
David~Van~Dantzig.
\newblock Zur topologischen {A}lgebra. {III}. {B}rouwersche und {C}antorsche
  {G}ruppen.
\newblock {\em Compositio Math.}, 3:408--426, 1936.

\bibitem{Wi94}
George~Willis.
\newblock The structure of totally disconnected, locally compact groups.
\newblock {\em Math. Ann.}, 300(2):341--363, 1994.

\end{thebibliography}

\end{document}